\documentclass[11pt]{amsart}
\usepackage{amsthm,amsmath,graphicx,tikz,subfig}
\usetikzlibrary{calc,decorations.markings,matrix,arrows,intersections}
\newtheorem{theorem}{Theorem}[section]

\newtheorem{lemma}[theorem]{Lemma}
\newtheorem{corollary}[theorem]{Corollary}
\newtheorem{conjecture}[theorem]{Conjecture}

\newcommand{\binomial}[2]{\left(\begin{array}{c} #1 \\ #2 \end{array}\right)}

\tikzset{tags/.style={postaction={decorate,
    decoration={markings,mark=at position .4cm with
    {\draw (-.035,-.07) -- (.035,.07) -- (-.035,.07) -- (.035,-.07) -- cycle;}}}}}
\tikzset{tage/.style={postaction={decorate,
    decoration={markings,mark=at position -.4cm with
    {\draw (-.035,-.07) -- (.035,.07) -- (-.035,.07) -- (.035,-.07) -- cycle;}}}}}

\begin{document}

\title{Non-zero integral friezes}
\author{Bruce Fontaine}
\email{bfontain@math.cornell.edu}
\address{ Bruce Fontaine \\
          Department of Mathematics \\
          310 Malott Hall \\
          Cornell University \\
          Ithaca, NY USA.}
\begin{abstract}
  We study non-zero integral friezes for Dynkin types $A_n$, $B_n$, $C_n$,
  $D_n$ and $G_2$. These differ from standard Coxeter-Conway (positive) friezes
  by allowing any non-zero integer to appear. In each case we show that there
  are either $1$, $2$ or $4$ times as many non-zero friezes as positive friezes.
	This is a first step for considering friezes over general rings of integers.
\end{abstract}

\maketitle

\section{Introduction}

In 1973 Coxeter and Conway \cite{CC} studied what they called friezes. These
were grids of strictly positive (non-zero) integers on the plane satisfying a
determinant condition. For instance the following is a frieze height $3$:

$$\begin{tabular}{cccccccccccccc}
   & 1 & & 1 & & 1 & & 1 & & 1 & & 1 & & 1 \\
  4 & & 1 & & 2 & & 2 & & 2 & & 1 & & 4 & \\
   & 3 & & 1 & & 3 & & 3 & & 1 & & 3 & & 3 \\
  2 & & 2 & & 1 & & 4 & & 1 & & 2 & & 2 & \\
   & 1 & & 1 & & 1 & & 1 & & 1 & & 1 & & 1
\end{tabular}$$

Here the height measures the number of rows between the two rows of $1$s. In the
above figure, every two by two diamond \begin{tabular}{ccc} & $a$ & \\ $b$ & &
$c$ \\ & $d$ & \end{tabular} satisfies $ad+1=bc$ or $ad-bc=1$.

Coxeter-Conway showed the following was true:

\begin{theorem}\cite{CC}
  There are $C(k+1)$ friezes of height $k$ where $C(k+1)$ is the $k+1$-st
  Catalan number.
\end{theorem}

They discovered this by connecting friezes to triangulations of polygons. They
observed that a frieze of height $k$ corresponds the labeling of each possible
arc that can occur in a triangulation of a $k+3$-gon with a positive integer.
They required that the external arcs be labeled $1$ and that for any $4$ arcs
forming a quadrilateral, the labels of those arcs and the two internal diagonal
arcs satisfy the the Ptolemy relation: that the sum of product of opposite sides
is the product of the diagonals.

If one considers the zigzag triangulation, the labels recover a zigzag column of
the frieze. Moreover, successive columns are recovered by rotating the zigzag.
Finally, they show that for any such assignment of integers to polygon arcs, the
set of arcs labeled $1$ form a triangulation. And conversely if one starts with
a triangulation where all arcs are labeled $1$, the remaining arcs are uniquely
determined by the Ptolemy relations and are all positive integers.

More recently, it was noted by Caldero \cite{ARS10}, that there is a deep
connection to cluster algebras. Indeed, due to the above connection to
triangulations, a frieze can be interpreted as an evaluation of the cluster
variables of the Dynkin type $A_k$ cluster algebra such that each cluster
variable is sent to a positive integer. An alternative, but equivalent
formulation is that a frieze is a ring homomorphism from the cluster algebra
to $\mathbb{Z}$ such that each cluster variable is mapped to a positive integer.

Constructing a frieze is simple: pick a cluster in the cluster algebra and
assign to each variable in the cluster the value $1$. Then due to the positivity
of the exchange relations, each other cluster variable must evaluate to a
positive number. Moreover, due to the Laurent phenomenon (every cluster variable
can be expressed as a Laurent polynomial with integral coefficients in terms of
the cluster variables in a single cluster), every other cluster variable
evaluates to an integer.

In this paper we will relax the definition by allowing negative integers to
appear. Given a cluster algebra, a \textbf{non-zero integral frieze} is an
assignment of a non-zero integer to each cluster variable satisfying the
mutation rules. We will use the term positive frieze to denote a frieze which
has all cluster variables evaluated to positive integers. Alternatively it is a
ring homomorphism from the cluster algebra to $\mathbb{Z}$ that sends each
cluster variable to a non-zero integer.

The following 3 theorems are extensions of the results in \cite{FP}:

\begin{theorem}
  For $A_k$, the number of non-zero integral friezes is twice the number of
  positive friezes when $k$ is odd and the same number when $k$ is even.
\end{theorem}

\begin{theorem}
  For $B_k$ and $C_k$, there are twice the number of non-zero integral friezes
  as positive friezes.
\end{theorem}

\begin{theorem}
  For $D_k$ there are $4$ times as many non-zero friezes when $k$ is even and
  $2$ times as many when $k$ is odd.
\end{theorem}

\begin{theorem}
  For $G_2$ there are $9$ non-zero integral friezes and they are all positive.
\end{theorem}

I would like to thank Pierre-Guy Plamondon, Allen Knutson and Richard Stanley
for useful discussions and the Sage mathematics community for their software and
its cluster algebra implementation.

\section{Admissible labelings of triangulations}\label{sec:admis}
Consider a triangulation of an $n$-gon whose arcs (including boundary arcs) are
labeled by $\pm 1$. A labeling is said to be \textbf{admissible} if for any
$4$-cycle, the product of the labels around the cycle is $1$.

The \textbf{boundary state} of a labeled triangulation is simply the set of
boundary labels considered with their cyclic order. It will be of interest to
understand the structure of admissible labelings. Specifically, we want to know
what the possible boundary states for an admissible labeling are and given such
a boundary state, which admissible labelings have that boundary state.

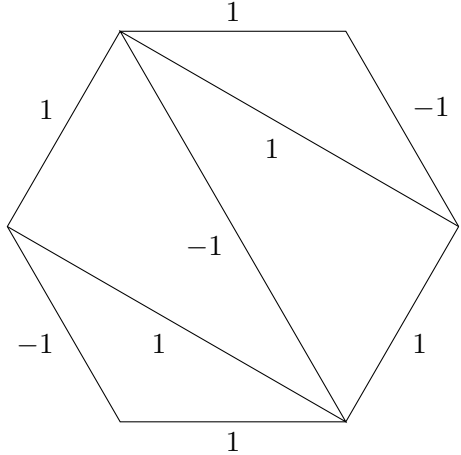
\begin{figure}
\begin{tikzpicture}
\draw (0:3) -- node[midway,above right] {$-1$} (60:3) -- node[midway, above] {$1$} (120:3) -- node[midway, above left] {$1$} (180:3) -- node[midway, below left] {$-1$} (240:3) -- node[midway, below] {$1$} (300:3) -- node[midway, below right] {$1$} (0:3);
\draw (0:3) -- node[midway, below left] {$1$} (120:3) -- node[midway, below left] {$-1$} (300:3) -- node[midway, below left] {$1$} (180:3);
\end{tikzpicture}
\caption{An admissible labeling of a $6$-gon with boundary state $-1,1,1,-1,1,1$}
\end{figure}

\begin{lemma}\label{lem:admisinv}
  Given an admissible labeling of a $n$-gon where $n$ is even, the labeling
  formed by negating all arcs of even length (i.e. the end points of the arc
  have an odd number of vertices between them) gives another admissible
  labeling.
\end{lemma}

\begin{proof}
  Given a 4 cycle in the triangulation, since the set of lengths of these arcs
  adds to $n$, which is even, it follows that there are an even number of arcs
  of even length in the 4 cycle. But then the negation of the signs on the arcs
  of even length cancels when we take the product of all the labels on the 4
  cycle. Thus the product of the labels is still one for each 4 cycle.
\end{proof}

Note that this operation acts as an involution on the set of admissible
labellings.

\begin{theorem}
  Given a triangulation of an $n$-gon and a boundary state consisting of $\pm
  1$, it extends to an admissible labeling of the triangulation if and only if
  $n$ is odd or $n$ is even and the product of the boundary state is $1$. More
  over, when $n$ is odd, there is a unique extension to an admissible labeling
  of the triangulation and when $n$ is even and the product is $1$, there are
  exactly two extensions related by the involution of lemma~\ref{lem:admisinv}
\end{theorem}

\begin{proof}
  In the case of $n=3$, this is true, since the triangulation is trivial, there
  are only the boundary labels and every state leads to an admissible labeling.

  When $n=4$, since the boundary is a $4$-cycle, a labeling is admissible if and
  only if the product of the boundary labels is $1$. The single internal arc of
  the triangulation can be labeled $\pm 1$ giving rise to exactly two admissible
  labelings for each boundary state. Since the internal arc is of length $2$,
  these two are related by lemma~\ref{lem:admisinv}.

  For general $n>4$ we proceed via induction. Note that any triangulation must
  have at least one arc of length $2$. Consider cutting the triangle that the
  arc forms off the triangulation. The result is a triangulation of an
  $n-1$-gon. Now we break into two cases, based on the parity of $n$.

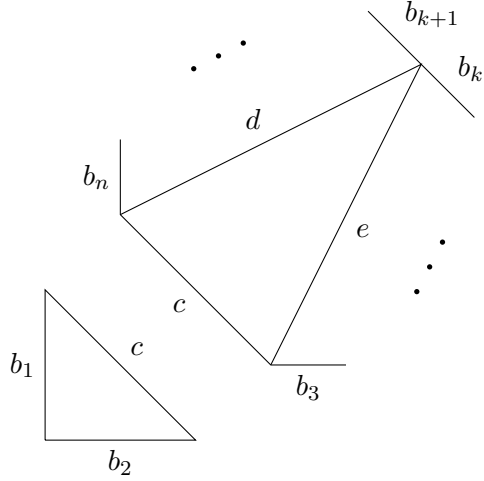
$$\begin{tikzpicture}
  \draw (0,0) -- node[midway, left] {$b_1$} (0,2) -- node[midway, above right] {$c$} (2,0) -- node[midway, below] {$b_2$} (0,0);
  \draw (1,4) -- node[midway, left] {$b_n$} (1,3) -- node[midway, below left] {$c$} (3,1) -- node[midway, below] {$b_3$} (4,1);
  \fill ($ 0.5*(4.293,5.707) + 0.5*(1,4) + 0.2*(-1.707,1.293) $) circle (.035);
  \fill ($ 0.4*(4.293,5.707) + 0.6*(1,4) + 0.2*(-1.707,1.293) $) circle (.035);
  \fill ($ 0.6*(4.293,5.707) + 0.4*(1,4) + 0.2*(-1.707,1.293) $) circle (.035);
  \fill ($ 0.5*(5.707,4.293) + 0.5*(4,1) + 0.2*(1.293,-1.707) $) circle (.035);
  \fill ($ 0.4*(5.707,4.293) + 0.6*(4,1) + 0.2*(1.293,-1.707) $) circle (.035);
  \fill ($ 0.6*(5.707,4.293) + 0.4*(4,1) + 0.2*(1.293,-1.707) $) circle (.035);
  \draw (1,3) -- node[midway, above left] {$d$} (5,5) -- node[midway, below right] {$e$} (3,1);
  \draw (4.293,5.707) -- node[midway, above right] {$b_{k+1}$} (5,5) -- node[midway, above right] {$b_k$} (5.707,4.293);
\end{tikzpicture}$$

  Let $n$ be odd and pick an arbitrary boundary state $a_i$. If there exists an
  admissible labeling for this boundary state, its restriction to the $n-1$-gon
  will be an admissible labeling. Since $n-1$ is even, by induction the only way
  this can happen is if $c=\prod_{i=3}^nb_i$, since $c,b_3,\dots,b_n$ is the
  boundary state for the $n-1$-gon. By induction there are two possible
  admissible labelings that extend this boundary state. Note that exactly one of
  the two edges labeled $e$ and $d$ is even length and the other is odd. Since
  the two labelings are related by lemma~\ref{lem:admisinv} it follows that in
  one labeling $de=1$ and in the other $de=-1$. Since $b_1b_2$ is fixed, only
  one of the two satisfies $de=b_1b_2$. I.e. there is one admissible labeling
  for the boundary state $b_i$.

  If $n$ is even, once again start with a boundary state $b_i$. Given a choice
  of $c$, by induction it uniquely extends to an admissible labeling of the
  $n-1$-gon. Thus we have at most two admissible labelings of the $n$-gon with
  boundary $b_i$. In order for this labeling to be admissible we must have
  $b_1b_2=de$.

  Let $g=\prod_{i=2}^kb_i$ and $h=\prod_{i=k+1}^nb_i$. Consider the edges labeled
  $d$ and $e$. They are either both of even length or odd length. If they are of
  odd length, consider the polygon formed from $e$ and the boundary edges $b_2$
  to $b_k$ and the polygon formed from $d$ and the boundary edges $b_{k+1}$ to
  $b_n$. These are both even, so $eg=dh=1$. Otherwise if they are of even length
  consider the polygons formed by $e$, $c$ and $b_{k+1}$ to $b_n$ and by $d$,
  $c$ and $b_2$ to $b_k$. Since they are both even, we have $ech=dcg=1$. In
  either case this gives $de=gh$. Note that this is independent of our choices
  of $c$.

  Thus if $\prod b_i\neq 1$, then $b_1b_2\neq gh=de$ and there are no
  extensions. Otherwise there are exactly $2$ as determined by the sign of $c$.
\end{proof}

\section{Non-zero integral $A_{n-3}$ friezes}

In \cite{FZ1}, it is shown that the clusters in the type $A_{n-3}$ cluster
algebra are associated to the triangulations of $n$-gon with the boundary arcs
labeled $1$. The cluster variables are the arcs in triangulations and the
relations (mutations) between them are the Ptolemy relation on quadrilaterals.

A \textbf{non-zero integral frieze} is an evaluation of the cluster variables
such that each cluster variable evaluates to a non-zero integer. In the language
of triangulations of an $n$-gon, this corresponds to labeling each internal arc
with a non-zero integer such that the labeling satisfies the Ptolemy relation.

In order to understand non-zero integral friezes of type $A_{n-3}$ we will need
to consider friezes of a slightly enlarged cluster algebra $FA_{n-3}$ where we
add a frozen variable for each boundary arc. The value of the frozen variables
corresponding to the boundary arcs will be called the \textbf{boundary state}
and thus the non-zero integral friezes of type $FA_{n-3}$ with boundary state
$1,\dots,1$ will be the non-zero integral friezes of type $A_{n-1}$.

\begin{theorem}
  There are $2C(n-2)$ non-zero integral $A_{n-3}$ friezes when $n$ is even and
  $C(n-2)$ when $n$ is odd.
\end{theorem}

Note that the the $C(n-2)$ positive friezes of Coxeter-Conway are the portion of
non-zero integral friezes that are positive. When $n$ is even one can obtain the
remaining $C(n-2)$ non-zero integral friezes from the positive integral friezes
in the following way:

\begin{lemma}\label{lem:neg}
  If $n$ is even, there exists a non trivial involution $\sigma$ on the set of
  non-zero integral $FA_{n-3}$ friezes: Negate all labels in the frieze on arcs
  of even length.
\end{lemma}

\begin{figure}
\begin{tikzpicture}
\draw (180:2) -- node[midway, above left, scale=0.7] {$1$} (135:2) -- node[midway, above, scale=0.7] {$1$} (90:2) -- node[midway, above, scale=0.7] {$1$} (45:2) -- node[midway, above right, scale=0.7] {$1$} (0:2);
\draw (180:2) -- node[pos=0.25, above, scale=0.7] {$1$} (90:2) -- node[pos=0.75, above, scale=0.7] {$2$} (0:2);
\draw (135:2) -- node[midway, below, scale=0.7] {$2$} (45:2);
\draw (180:2) -- node[midway, below, scale=0.7] {$1$} (45:2);
\draw (135:2) -- node[midway, below, scale=0.7] {$3$} (0:2);
\end{tikzpicture} $\rightarrow$ \begin{tikzpicture}
\draw (180:2) -- node[midway, above left, scale=0.7] {$1$} (135:2) -- node[midway, above, scale=0.7] {$1$} (90:2) -- node[midway, above, scale=0.7] {$1$} (45:2) -- node[midway, above right, scale=0.7] {$1$} (0:2);
\draw (180:2) -- node[pos=0.25, above, scale=0.7] {$-1$} (90:2) -- node[pos=0.75, above, scale=0.7] {$-2$} (0:2);
\draw (135:2) -- node[midway, below, scale=0.7] {$-2$} (45:2);
\draw (180:2) -- node[midway, below, scale=0.7] {$1$} (45:2);
\draw (135:2) -- node[midway, below, scale=0.7] {$3$} (0:2);
\end{tikzpicture}
\caption{Transforming a positive integral frieze into a non-zero frieze containing negative entries.}
\end{figure}
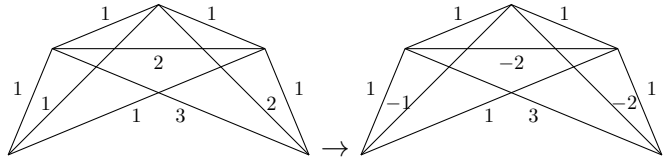

\begin{proof}
  Beginning with a non-zero integral frieze $F$, consider any quadrilateral in
  the triangulation of the $n$-gon. Since $n$ is even the parity of the lengths
  of the arcs of the quadrilateral fall into 4 general cases: all odd, all even,
  (odd,even,odd,even) and (odd,odd,even,even).

  In each case the sign difference between $F$ and $\sigma(F)$ cancels in the
  Ptolemy relation, either because the summands are of total degree 2 and each
  term in a monomial is negated, or because one term in each of the three
  summands is negated.
\end{proof}

\begin{lemma}\label{lem:anarc}
  Consider a non-zero integral $FA_{n-3}$ frieze. Let $\{b_i\}$ be its boundary
  state and $a_i$ the labels of the arcs of length $2$ such that arcs labeled
  $a_i$, $b_i$ and $b_{i+1}$ form a triangle. Then there exists an index $i$
  such that $|a_i|<|b_i|+|b_{i+1}|$.
\end{lemma}

Note that this proof is a variation of that of Coxeter-Conway in the case of
positive friezes.

\begin{proof}
  Fix a vertex of the polygon and let $f_i$ denote the integers labeling the
  $n-1$ arcs at a fixed vertex ordered clockwise. Let $b_i$ be the boundary
  state, starting clockwise from the fixed vertex and $a_i$ as in the statement.

$$
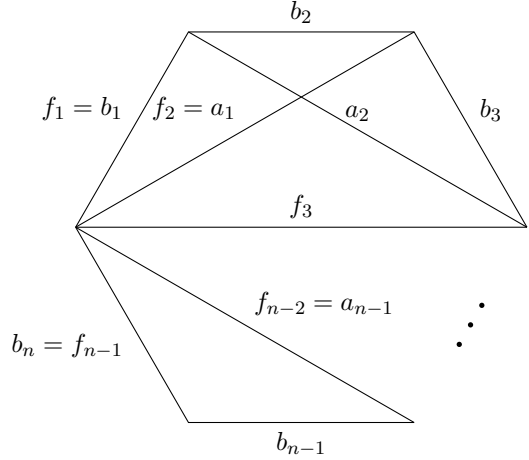
\begin{tikzpicture}
\draw (180:3) -- node[midway, above left, scale=0.9] {$f_1=b_1$} (120:3) -- node[midway, above, scale=0.9] {$b_2$} (60:3) -- node[midway, above right, scale=0.9] {$b_3$} (0:3);
\draw (120:3) -- node[midway, above, scale=0.9] {$a_2$} (0:3);
\draw (180:3) -- node[midway, below left, scale=0.9] {$b_{n}=f_{n-1}$} (240:3) -- node[midway, below, scale=0.9] {$b_{n-1}$} (300:3);
\draw (180:3) -- node[midway, above left, scale=0.9] {$f_2=a_1$} (60:3);
\draw (180:3) -- node[midway, above, scale=0.9] {$f_3$} (0:3);
\draw (180:3) -- node[midway, above right, scale=0.9] {$f_{n-2}=a_{n-1}$} (300:3);
\fill ($ 0.4*(0:3)+0.6*(300:3) $) circle (.035);
\fill ($ 0.5*(0:3)+0.5*(300:3) $) circle (.035);
\fill ($ 0.6*(0:3)+0.4*(300:3) $) circle (.035);

\end{tikzpicture}
$$

  Suppose the lemma is not true, then $|a_i|\geq|b_i|+|b_{i+1}|$ for all $i$.
  For each $i$ the cluster mutation relations are
  $f_{i-1}b_{i+1}+f_{i+1}b_{i}=a_if_i$. Taking absolute value of both sides:
  $$|f_{i-1}b_{i+1}+f_{i+1}b_{i}|=|a_i||f_i|.$$ Applying triangle inequality to
  the left side and using $|a_i|\geq|b_i|+|b_{i+1}|$ on the right, we get
  $$|f_{i-1}||b_{i+1}|+|f_{i+1}||b_i|\geq|b_i||f_i|+|b_{i+1}||f_i|.$$

  Rearranging we have
  $$\frac{|f_{i+1}|-|f_i|}{|b_{i+1}|}\geq\frac{|f_i|-|f_{i-1}|}{|b_i|}.$$
  Thus the sequence $\frac{|f_i|-|f_{i-1}|}{|b_i|}$ from $i=2$ to $n-1$ is
  increasing.

  When $i=2$, we have
  $$\frac{|f_2|-|f_1|}{|b_2|}=\frac{|a_1|-|b_1|}{|b_2|}\geq\frac{|b_2|}{|b_2|}=1.$$

  On the other hand when $i=n-1$, we have
  $$\frac{|f_{n-1}|-|f_{n-2}|}{|b_{n-1}|}=\frac{|b_n|-|a_{n-1}|}{|b_{n-1}|}\leq\frac{-|b_{n-1}|}{|b_{n-1}|}=-1.$$

  But this is impossible, no increasing sequence can start at $1$ and end at
  $-1$. Thus we must have $|a_i|<|b_i|+|b_{i+1}|$ for some $i$.
\end{proof}

\begin{corollary}
  Every non-zero integral $FA_{n-3}$ frieze with boundary state composed of
  $\pm 1$ contains a triangulation whose arcs are all labeled $\pm 1$.
\end{corollary}

We can now prove the main result of this section:

\begin{proof}
  By the above corollary, since a non-zero integral $A_{n-3}$ frieze is a
  non-zero integral $FA_{n-3}$ frieze there is a triangulation whose arcs are
  labeled $\pm 1$. Since no arc in the frieze is labeled $0$, by the Ptolemy
  relation the configurations in Figure~\ref{fig:obs} are impossible. This means
  that the labeling is admissible. Thus the number of friezes is at most the
  product of the number of admissible labelings and the number of triangulations
  $C(n-2)$.

  But, we already know of $C(n-2)$ non-zero integral friezes: the set of
  positive integral friezes. Thus, in the case that $n$ is odd, the number of non-zero
  integral friezes is $C(n-2)$. In the case of $n$ even, for each positive
  integral frieze $F$, $\sigma(F)$ is also non-zero integral frieze, and since
  $\sigma$ is an involution we have at least $2C(n-2)$ non-zero integral
  friezes.
\end{proof}

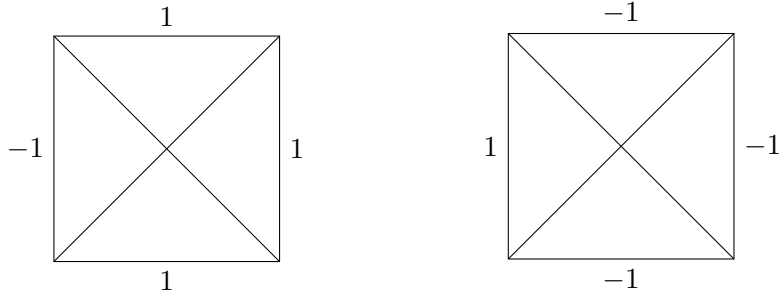
\begin{figure}
\begin{tikzpicture}
\draw (0,0) -- (3,0) -- (3,3) -- (0,3) -- cycle;
\draw (0,0) -- (3,3);
\draw (3,0) -- (0,3);
\node at (1.5,0) [below] {$1$};
\node at (0,1.5) [left] {$-1$};
\node at (1.5,3) [above] {$1$};
\node at (3,1.5) [right] {$1$};
\end{tikzpicture}\hspace{20 mm}
\begin{tikzpicture}
\draw (0,0) -- (3,0) -- (3,3) -- (0,3) -- cycle;
\draw (0,0) -- (3,3);
\draw (3,0) -- (0,3);
\node at (1.5,0) [below] {$-1$};
\node at (0,1.5) [left] {$1$};
\node at (1.5,3) [above] {$-1$};
\node at (3,1.5) [right] {$-1$};
\end{tikzpicture}
\caption{Obstructions for non-zero friezes\label{fig:obs}}
\end{figure}

\section{Non-zero integral $D_n$ friezes}

In \cite{FST}, it is shown that the cluster algebra of type $D_n$ can be thought
of in terms of (tagged) arcs and triangulations for a punctured $n$-gon. Thus, in
parallel with the $A_n$ case, a non-zero integral $D_n$ frieze is equivalent to
a labeling of the (possibly tagged) arcs of a punctured $n$-gon by non-zero
integers such the boundary is labeled $1$ and the relations in
figure~\ref{fig:dnrel} hold.

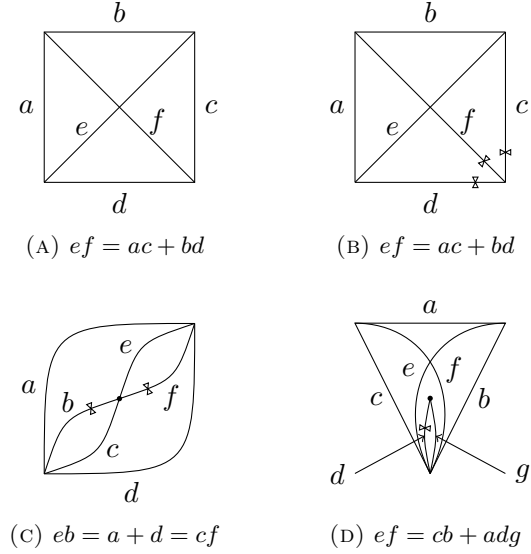
\begin{figure}
\subfloat[$ef=ac+bd$\label{fig:dnrel:a}]{
\begin{tikzpicture}
\path (-1,0) rectangle(3,0);
\draw (0,0) -- node[midway, left] {$a$} (0,2) -- node[midway, above] {$b$} (2,2) -- node[midway, right] {$c$} (2,0) -- node[midway, below] {$d$} (0,0);
\draw (0,0) -- node[near start, above] {$e$} (2,2);
\draw (2,0) -- node[near start, above] {$f$} (0,2);
\end{tikzpicture}}
\subfloat[$ef=ac+bd$\label{fig:dnrel:b}]{
\begin{tikzpicture}
\path (-1,0) rectangle(3,0);
\draw (0,0) -- node[midway, left] {$a$} (0,2) -- node[midway, above] {$b$} (2,2);
\draw[tage] (2,2) -- node[midway, right] {$c$} (2,0);
\draw[tags] (2,0) -- node[midway, below] {$d$} (0,0);
\draw (0,0) -- node[near start, above] {$e$} (2,2);
\draw[tags] (2,0) -- node[near start, above] {$f$} (0,2);
\end{tikzpicture}} \\
\subfloat[$eb=a+d=cf$\label{fig:dnrel:c}]{
\begin{tikzpicture}
\path (-1,0) rectangle(3,0);
\draw (0,0) .. controls (0,2) .. node[near start, left] {$a$} (2,2);
\draw (0,0) .. controls (2,0) .. node[near start, below] {$d$} (2,2);
\draw[tage] (0,0) .. controls (0.25,0.75) .. node[midway,above] {$b$} (1,1);
\draw (0,0) .. controls (0.75,0.25) .. node[midway,right] {$c$} (1,1);
\draw (1,1) .. controls (1.25,1.75) .. node[midway,left] {$e$} (2,2);
\draw[tags] (1,1) .. controls (1.75,1.25) .. node[midway,below] {$f$} (2,2);
\fill (1,1) circle(.035);
\end{tikzpicture}}
\subfloat[$ef=cb+adg$\label{fig:dnrel:d}]{
\begin{tikzpicture}
\path[name path=pline] (-2,0.5) rectangle(2,0.5);
\draw (0,0) node[below] {$\phantom{q}$} -- node[midway,left] {$c$}  (-1,2) -- node[midway,above] {$a$} (1,2) -- node[midway,right] {$b$} (0,0);
\draw (0,0) .. controls (-0.5,1) and (0,2) .. node[midway,left] {$e$} (1,2);
\draw (0,0) .. controls (0.5,1) and (0,2) .. node[midway,right] {$f$} (-1,2);
\draw[name path=pleft,tage] (0,0) .. controls (-0.1,0.5) .. (0,1);
\draw[name path=pright] (0,0) .. controls (0.1,0.5) .. (0,1);
\draw[name intersections={of=pleft and pline},->] (-1,0) node[left] {$d$} -- (intersection-1);
\draw[name intersections={of=pright and pline},->] (1,0) node[right] {$g$}-- (intersection-1);
\fill (0,1) circle(.035);
\end{tikzpicture}}
\caption{$D_n$ relations\label{fig:dnrel}}
\end{figure}

In \cite{FP} we determined that the number of positive friezes for $D_n$ was:

$$\sum_{m=1}^n d(m)\binomial{2n-m-1}{n-m}$$

Each positive integral $D_n$ frieze has a canonical triangulation, containing
all arcs between boundary vertices labeled $1$ and enough (untagged) spokes to
form a triangulation. The $m$ spokes all have the same label, an integer that
divides $m$.

\begin{theorem}\label{thm:dnmain}
  The number of non-zero integral friezes of type $D_n$ is
  $$4\sum_{m=1}^n d(m)\binomial{2n-m-1}{n-m}$$ when $n$ is even and
  $$2\sum_{m=1}^n d(m)\binomial{2n-m-1}{n-m}$$ when $n$ is odd.
\end{theorem}

Similarly with the non-zero integral $A_{n-3}$ friezes, we will need to consider
non-zero integral friezes over an enlarged cluster algebra $FD_n$ where one adds
in a frozen variable for each boundary arc of the punctured $n$-gon. In this
case we have the following involutions:

\begin{lemma}\label{lem:dninv:1}
  Given a non-zero integral $FD_n$ frieze, negating all spoke arcs (tagged and
  non-tagged) gives a non-zero integral frieze. Call this $\sigma_1$.
\end{lemma}

\begin{proof}
  Given $F$, a non-zero integral $FD_n$ frieze, consider a tagged triangulation
  and an arc in that triangulation. The mutation relation for this arc in the
  triangulation must be one of the relations in figure~\ref{fig:dnrel}. If the
  puncture is not one of the vertices, then $F$ and $\sigma_1(F)$ agree on all
  arcs involved in the relation and the relation is true for $\sigma_1(F)$.

  On the other hand, if one of the vertices involved is the puncture, when one
  examines the relations in Figure~\ref{fig:dnrel}, we see that either all terms
  contain one spoke and one non-spoke variable, in which case the relation for
  $\sigma_1(F)$ is true since it is the negation of that for $F$. Otherwise the
  spoke variables appear an even number of times in each multiplicand and again
  the $\sigma_1(F)$ relation follows from the relation in $F$.
\end{proof}

\begin{lemma}\label{lem:dninv:2}
  When $n$ is even, given a non-zero integral $FD_n$ frieze, fix an untagged
  spoke and negate every second untagged spoke. Negate the opposite set of
  tagged spokes and all arcs of even length between boundary vertices. This
  gives a non-zero integral $FD_n$ frieze. Let $\sigma_2$ denote this
  involution.
\end{lemma}

Note that that there are really two involutions here, depending on which spoke
one fixes, but they are related by $\sigma_1$. That is $\sigma_2\circ\sigma_1$
is the other involution.

\begin{proof}
  Let $F$ be a $FD_n$ frieze. Let $F(a)$ denote the value of $F$ on edge $a$.

  For relation A and B: With out loss of generality, we may assume
	$F(d)=(\sigma_2(F))(d)$ otherwise we consider $\sigma_1(\sigma_2(F))$. This
	case then follows from the $A_{n-3}$ case: if none of the arcs are spokes it is
	clear. Otherwise, one considers the length a spoke to be $1$ plus the distance
	from $d$, then $F$ and $\sigma_2(F)$ differ in sign	on exactly the edges of
	even length. The boundary of the relation also has an	even total length.

  For relation C: By symmetry we can just verify the first relation. Since $n$
  is even, it follows that the length of $a$ and $d$ have the same parity. If
  that parity is even, $\sigma_2(F)(a)=-F(a)$ and $\sigma_2(F)(d)=-F(d)$,
	so $\sigma_2(F)(a)+\sigma_2(F)(d)=-(F(a)+F(d))$. Considering $e$ and $b$, we
	have $\sigma_2(F)(e)\sigma_2(F)(b)=-F(e)F(b)$ since one is of even length
	and the other odd. If the parity is odd, then
	$\sigma_2(F)(a)+\sigma_2(F)(d)=F(a)+F(d)$ and 
	$\sigma_2(F)(e)\sigma_2(F)(b)=F(e)F(b)$ since the lengths of $e$ and $b$
	are both even or both odd. In either case we have 
	$\sigma_2(F)(a)+\sigma_2(F)(d)=\sigma_2(F)(e)\sigma_2(F)(b)$ when
	$F(a)+F(b)=F(e)F(b)$.

  For relation D: First note that $\sigma_2(F)(d)\sigma_2(F)(g)=-F(d)F(g)$.
	Suppose that $c$ and $b$ are both even or both odd. Then
	$\sigma_2(F)(e)\sigma_2(F)(f)=F(e)F(f)$ and
	$\sigma_2(F)(c)\sigma_2(F)(b)=F(c)F(b)$. Then $a$ is of even length, so
	$\sigma_2(F)(a)=-F(a)$. Thus 
	$\sigma_2(F)(a)\sigma_2(F)(d)\sigma_2(F)(g)=F(a)F(d)F(g)$.
	On the other hand, the lengths of $c$ and $b$ have opposite parity, we have
	$\sigma_2(F)(e)\sigma_2(F)(f)=-F(e)F(f)$ and
	$\sigma_2(F)(c)\sigma_2(F)(b)=-F(c)F(b)$. Then $a$ is of odd length, so
	$\sigma_2(F)(a)=F(a)$. Thus in either case when the relation is satisfied under
	$F$ it is satisfied under $\sigma_2(F)$.
\end{proof}

Using these two involutions, we can generate $4$ non-zero integral friezes from
each positive frieze when $n$ is even and $2$ when $n$ is odd. All that remains
is to prove that we have no more. We will follow the method of proof from
\cite{FP}, but adjusted to allow arcs labeled by $-1$. From this point onwards

\begin{lemma}\label{lem:dntriang}
  Given a non-zero integral $FD_n$ frieze for $n\geq 2$, with a boundary state
  consisting of $\pm 1$, there exists a triangulation consisting of arcs between
  boundary vertices labeled $\pm 1$ and $m$ untagged spokes labeled by $\pm d$
  for some non-zero integer $d$.
\end{lemma}

\begin{proof}
  When $n=2$ we are in the situation of figure~\ref{fig:dnrel:c}. We cannot have
  $a=-d$, otherwise one of $e$ or $b$ will be $0$. Thus $a=d$, so $|a+d|=2$ and
  it follows that $|eb|=2=|cf|$. If $|e|=|c|$, then since $e|2$ we take the arcs
  labeled $e$ and $c$ for the triangulation. Otherwise $|e|\neq |c|$, in which
  case with out loss of generality, $|e|=2$ so $|b|=1$, and we take the arcs
  labeled $c$ and $b$, and consider $b$ to be a loop rather than a tagged arc.

  When $n>2$, we fall into two cases, either there exists an internal arc of
  length $2$ labeled $\pm 1$ or there are none. If there is one, cut the
  triangle subtended by the arc off, then we have a $FD_{n-1}$ frieze and by
  induction, it follows that the required triangulation exists.

  Otherwise all arcs of length $2$ have labels of norm larger than $1$. Consider
  the triangulation by untagged spokes. Suppose the $i$-th spoke (modulo $n$)
  has label $f_i$ and the arc of length $2$ crossing the $i$-th spoke has label
  $a_i$. The boundary state is given by $b_i$, where $b_i$ labels the boundary
  arc joining vertex $i-1$ and $i$.

  Then we have: $$a_if_i=b_{i+1}f_{i-1}+b_if_{i+1}.$$ Taking norms and using
  triangle inequality, (note that $|b_i|=1$) we have
  $$|a_i||f_i|\leq |f_{i-1}|+|f_{i+1}|.$$ And since $|a_i|\geq 2$, we have
  $$2|f_i|\leq |f_{i-1}|+|f_{i+1}|,$$ or $$|f_i|-|f_{i-1}|\leq|f_{i+1}|-|f_i|.$$
  The sequence $|f_i|-|f_{i-1}|$ is cyclic and monotone, so it is constant. Then
  the sequence $|f_i|$ is monotone and also cyclic, so it is also constant.
\end{proof}

Given a non-zero integral $FD_n$ frieze and a triangulation of the punctured
$n$-gon, we associate a \textbf{sign configuration}. It is a labeling of the
triangulation by $\pm 1$ where an edge is labeled $1$ if the corresponding
label in the frieze is positive and $-1$ when the label is negative.

\begin{lemma}
  Given a non-zero integral $FD_n$ frieze with boundary $\pm 1$ and the
  triangulation from lemma~\ref{lem:dntriang}, the associated sign configuration
  is admissible in the sense of section~\ref{sec:admis} when we cut along one of
  the untagged spokes of the triangulation to obtain a triangulation of a
  $n+2$-gon.
\end{lemma}

\begin{proof}
  Note that any quadrilateral in the $n+2$-gon either comes from a honest 
	quadrilateral in the punctured $n$-gon or in the case where the is only one
	untagged spoke, from the following configuration (self folded triangle) where
	the cut is along $c$:

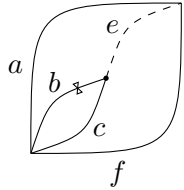
$$\begin{tikzpicture}
\path (-1,0) rectangle(3,0);
\draw (0,0) .. controls (0,2) .. node[near start, left] {$a$} (2,2);
\draw (0,0) .. controls (2,0) .. node[near start, below] {$f$} (2,2);
\draw[tage] (0,0) .. controls (0.25,0.75) .. node[midway,above] {$b$} (1,1);
\draw (0,0) .. controls (0.75,0.25) .. node[midway,right] {$c$} (1,1);
\draw[dashed] (1,1) .. controls (1.25,1.75) .. node[midway,left] {$e$} (2,2);
\fill (1,1) circle(.035);
\end{tikzpicture}$$

  Quadrilaterals in the punctured $n$-gon come in two classes as well: those
  whose vertices are all boundary vertices and those which have the puncture
  as one boundary vertex. In the first case, all arcs of the quadrilateral are
  already labeled $\pm 1$ and the two diagonals are non zero. It then follows
	from the relation in figure~\ref{fig:dnrel:a} that there are an even number
	of $-1$ labels on the quadrilateral and thus this quadrilateral is admissible.

  In the second case, the quadrilateral has two adjacent arcs labeled $\pm 1$
  and two adjacent arcs labeled $\pm d$ (the spokes). Since the product of the
  diagonals of the quadrilateral is nonzero, the sum of the product of the
	opposite labels of the quadrilateral are either $2d$ or $-2d$, so we cannot
	have an odd number of the arcs of the quadrilateral be negative and thus the
	quadrilateral is admissible

  In the self-folded case, we see that $a$ and $f$ are labeled $\pm 1$. If they
  have opposite sign, then $a+f=0=be$ so one of $b$ or $e$ would be zero. This
  is not the case, so $a$ and $f$ have the same sign. Cutting along $c$ revealed
  a quadrilateral with sides $a$, $f$, $c$ and $c$. Since the sign of $a$ and
  $f$ agree, the quadrilateral is admissible.
\end{proof}

In the case of a non-zero integral $D_n$ frieze, up to $\sigma_1$, we may assume
that the sign configuration has at least one spoke labeled $1$. Cutting along
this spoke gives triangulation of the $n+2$-gon and an admissible labeling with
boundary state $1$. By section~\ref{sec:admis} it follows that either all arcs
are labeled $1$ or $n$ is even and even length arcs are labeled $-1$.

The sign configuration falls into one of 4 categories:
\begin{enumerate}
  \item All arcs are labeled $1$.
  \item All arcs are labeled $1$ but spokes, which are labeled $-1$.
  \item $n$ is even and all odd length arcs are labeled $1$ and even length arcs
  are labeled $-1$. Counting from the spoke which was cut, spoke $0$ is positive
  and the sign of spoke $i$ is given by parity of the distance along the
  boundary from spoke $0$ with odd corresponding to $-1$ and even to $1$.
  \item The same as above, except the spokes have opposite labels.
\end{enumerate}

But now we see that if we apply a combination of $\sigma_1$ and $\sigma_2$, we
can turn a non-zero integral $D_n$ frieze a positive $D_n$ frieze. That is
$\sigma_2$ moves the sign from type $3$ and $4$ to $1$ and $2$ and $\sigma_1$
from $2$ to $1$. Once we have a positive sign configuration, the entire frieze
is positive since a single cluster is.

We can now prove theorem~\ref{thm:dnmain}:

\begin{proof}
	In the case that $n$ is odd, applying $\sigma_1$ to the set of positive $D_n$
	friezes shows that we have at least twice as many non-zero integral friezes of
	type $D_n$ as positive ones. On the other hand the above reasoning shows that
	any non-zero integral $D_n$ frieze can be brought to a positive one by
	applying $\sigma_1$. So there are exactly 
	$$2\sum_{m=1}^n d(m)\binomial{2n-m-1}{n-m}$$ non-zero integral $D_n$ friezes
	when $n$ is odd.
	
	Similarly when $n$ is even, since we have $\sigma_1$ and $\sigma_2$, we have
	$$4\sum_{m=1}^n d(m)\binomial{2n-m-1}{n-m}$$ non-zero integral $D_n$ friezes
	when $n$ is even.
\end{proof}

\section{$B_n$, $C_n$ and $G_2$}

The results for $B_n$, $C_n$ and $G_2$ can be derived from the results
$D_{n+1}$, $A_{2n-1}$ and $D_4$ respectively by using the folding methods of
\cite{DUP}. To summarize, if $\Delta$ is a Dynkin quiver and $G$ a group of
automorphisms of $\Delta$, then $\Delta/G$ is a valued quiver and the action of
$G$ lifts to the cluster algebra $A(\Delta)$. In the case of $B_n$, $G$ is
generated by the automorphism of $D_{n+1}$ which swaps the short arms. For
$C_n$, $G$ is generated by the automorphism of $A_{2n-1}$ which reflects the
diagram through the centre node. For $G_2$, $G$ is generated by the order $3$
automorphism of $D_4$ which rotates the diagram about its central node.

In this case \cite[Corollary 5.16]{DUP} shows that $A(\Delta/G)$ can be
identified with a subalgebra of $A(\Delta)/G$. More over
\cite[Corollary 7.3]{DUP} gives equality since $\Delta$ is Dynkin.
The projection $\pi:A(\Delta)\to A(\Delta)/G$ can then be thought of as a
surjective ring homomorphism from $A(\Delta)$ to $A(\Delta/G)$, which sends the
cluster variables of $A(\Delta)$ to the cluster variables of $A(\Delta/G)$ via a
quotient by $G$.

\begin{lemma}
  Let $\Delta$ by a Dynkin quiver and $G$ a group of automorphisms, then each
  non-zero integral $\Delta/G$ frieze gives rise to a $G$ invariant non-zero
  integral $\Delta$ frieze. More over, each non-zero integral $\Delta$ frieze
  that is $G$ invariant descends to a non-zero integral $\Delta/G$ frieze.
\end{lemma}

\begin{proof}
  We consider a non-zero integral $\Delta$ frieze to be a ring homomorphism from
  the cluster algebra $A(\Delta)$ to $\mathbb{Z}$ sending each cluster variable
  to a non-zero integer. Then a non-zero integral $\Delta/G$ frieze gives a
  $\Delta$ frieze by composition with $\pi$. Since $\pi$ is $G$ invariant so is
  the composition.

  Given a non-zero integral $\Delta$ frieze, if it is $G$ invariant, it descends
  to a ring homomorphism from $A(\Delta)/G$ to $\mathbb{Z}$, but this gives a
  non-zero integral $\Delta/G$ frieze.
\end{proof}

\begin{theorem}
  There are $2\binomial{2n}{n}$ non-zero integral $C_n$ friezes.
\end{theorem}

\begin{proof}
  Given a non-zero integral $C_n$ frieze, lift it to a non-zero integral
  $A_{2n-1}$ frieze. Since $2n-1$ is odd, if the frieze is not already positive
  applying $\sigma$ will make it so. Since $\sigma$ does not change the
  magnitude of the labels and makes all signs positive, the result will be $G$
  invariant. Since there are $\binomial{2n}{n}$ positive $G$ invariant frieze
  (\cite{FP}), it follows that the number of non-zero integral $C_n$ friezes is
  at most $2\binomial{2n}{n}$.

  On the other hand, the action of $G$ on $A_{2n-1}$ identifies exactly those
  arcs that are diametrically opposed. But these arcs have the same length so
  given a $G$ invariant $A_{2n-1}$ frieze, applying $\sigma$ to it will leave it
  $G$ invariant.

  Thus there are twice as many non-zero integral $C_n$ friezes as there are
  positive $C_n$ friezes.
\end{proof}

\begin{theorem}
  There are $2\sum_{m\leq\sqrt{n+1}}\binomial{2n-m^2+1}{n}$ frieze of type
  $B_n.$
\end{theorem}

\begin{proof}
  Given a non-zero integral $B_n$ frieze, lift it to a $G$ invariant $D_{n+1}$
  frieze. Applying some combination of $\sigma_1$ and $\sigma_2$ will give a
  positive frieze which will still be $G$ invariant (see above proof).

  The action of $G$ on $D_{n+1}$ identifies the tagged arc with its
  corresponding untagged arc. The result of applying $\sigma_2$ or
  $\sigma_1\circ\sigma_2$ to a $G$ invariant frieze will no longer be $G$
  invariant, since the labels of a tagged/untagged pair will differ by sign.
  On the other hand applying $\sigma_1$ leave the frieze $G$ invariant, since it
  negates all spokes arcs.

  Thus there are twice as many non-zero integral $B_n$ friezes as there are
  positive ones.
\end{proof}

\begin{theorem}
  Finally for the $G_2$ case, there are just $9$ non-zero integral friezes and
  they are all positive friezes.
\end{theorem}

\begin{proof}
  Suppose we have a non-zero integral $G_2$ frieze, we can lift it to a $G$
  invariant non-zero $D_4$ frieze. More over, by applying a combination of
  $\sigma_1$ and $\sigma_2$ we will obtain a positive $D_4$ frieze which must
  still be $G$ invariant.

	The $G$ action on $D_4$ leading to $G_2$ identifies the following arcs in a
	punctured $4$-gon:

$$
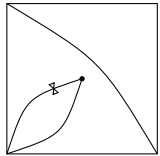
\begin{tikzpicture}
\draw (0,0) -- (0,2) -- (2,2) -- (2,0) -- cycle;
\draw[tage] (0,0) .. controls (0.25,0.75) .. (1,1);
\draw (0,0) .. controls (0.75,0.25) .. (1,1);
\draw (2,0) .. controls (1.25,1.25) .. (0,2);
\fill (1,1) circle(.035);
\end{tikzpicture}$$

  Applying $\sigma_2$ or $\sigma_1\circ\sigma_2$ to a positive $G$ invariant
  $D_4$ frieze gives a non-zero frieze which is not $G$ invariant since the
  labels of each pair of spokes have different signs. Applying $\sigma_1$
	negates all spoke labels but does not change the label of the diagonal arc in
	the above diagram, thus it is also not $G$ invariant.
\end{proof}

\section{Future directions}

By replacing positive by non-zero, we open the doors to consider the structure
of friezes over rings of integers. Interesting choices for such rings might
include $\mathbb{Z}[\zeta]$ for $\zeta$ a root of unity.

The results in this direction would look slightly different than those above.
For instance iterating lemma~\ref{lem:anarc} shows that starting with a boundary
state with bounded norm, one can obtain a triangulation of the $n$-gon where
arcs of length $i$ have labels bounded by the sum the norms of the labels on
one side of the initial polygon.

Indeed if we consider the case of $A_2$ we have the following:

\begin{theorem}
  There are $12$ nonzero $R=\mathbb{Z}[i]$ friezes of type $A_1$.
\end{theorem}

\begin{proof}
  Since the two cluster variables are $x_1$ and $\frac{2}{x_1}$, it suffices to
  find all elements $x_1\in R$ such that $2/x_1\in R$. The original non-zero
  integral friezes $\pm 1$ and $\pm 2$ work. We also have $\pm 1\pm i$, $\pm i$
  and $\pm 2i$.
\end{proof}

In each case it is possible to pick a triangulation such that the norm of the
label of the internal arcs is bounded (strictly) by $2$.

Irrespective of this problem, we should expect the following to be true:

\begin{conjecture}
  There are finitely many non-zero $\mathbb{Z}[\zeta]$ friezes of Dynkin type,
  for any $\zeta$ a root of unity.
\end{conjecture}

\end{document}